\documentclass[11pt, oneside]{amsart}
\usepackage{appendix}
\usepackage{graphicx}
\usepackage{chngpage}
\usepackage{multirow}
\usepackage{hyperref}
\usepackage[all,cmtip]{xy}
\usepackage{amssymb}
\theoremstyle{plain}
\newtheorem{theorem}{Theorem}[section]
\newtheorem{lemma}[theorem]{Lemma}
\newtheorem{proposition}[theorem]{Proposition}
\newtheorem{corollary}[theorem]{Corollary}

\newtheorem{problem}[theorem]{Problem}
\usepackage{fancyhdr}
\theoremstyle{definition}

\theoremstyle{remark}
\newtheorem{remark}[theorem]{Remark}

\begin{document}

   \title{On Castelnuovo theory and non-existence of smooth isolated curves in quintic threefolds}

\author{Xun Yu}
\address{Department of Mathematics, The Ohio State University, Columbus, OH, 43210-1174, USA}
\email{yu@math.ohio-state.edu,\; yxn100135@hotmail.com}

\maketitle

\begin{abstract}
We give some necessary conditions for a smooth irreducible curve $C\subset \mathbb{P}^4$ to be isolated in a smooth quintic threefold, and also find a lower bound for $h^1(\mathcal{N}_{C/{\mathbb{P}^4}})$. Combining these with beautiful results in Castelnuovo theory, we prove certain non-existence results on smooth curves in smooth quintic threefolds. As an application, we can prove Knutsen's list of examples of smooth isolated curves in general quintic threefolds is complete up to degree 9.
\end{abstract}
\section{Introduction}

We say a smooth projective curve $C$ is isolated in an ambient smooth projective variety $Y$ if $h^0(\mathcal{N}_{C/Y})=0$, where $\mathcal{N}_{C/Y}$ is the normal bundle of $C$ in $Y$. A Calabi-Yau threefold $Y$ has the nice property that the expected dimension of the deformation space of any l.c.i curve lying in $Y$ is zero. So it is quite reasonable to expect that Calabi-Yau threefolds contain isolated curves. More specifically, we can ask the following:
\begin{problem}
Let $d>0$ and $g\geq 0$ be integers. Does a general complete intersection Calabi-Yau (CICY) threefold (of a particular complete intersection type) contain a smooth isolated curve of degree $d$ and genus $g$ ? 
\end{problem}

Actually, many examples of smooth isolated curves in general CICY threefolds have been found by Knutsen's technique (Cf. [3], [10]). However, the highest genus $g$ known so far for which there exists a smooth isolated curve of genus $g$ in a general CICY threefold is 29. It is believed/conjectured that genera of isolated curves in CICY threefolds should be unbounded. 

In this paper we give some necessary conditions (Lemma 2.1, Theorem 2.7, and Theorem 2.9) for curves to be isolated in smooth quintic threefolds, and then using beautiful results (Theorem 3.1) in Castelnuovo theory we can prove certain non existence results (Theorem 3.5). As an application, we can prove Knutsen's list of examples of smooth isolated curves in general quintic threefolds is complete up to degree 9 (Corollary 3.6). It is also hoped that the non-existence results in this paper may be useful for people to get more existence results. For simplicity, this paper only considers non-existence of smooth isolated curves in quintic threefolds instead of CICY threefolds of various types.

\vspace{3mm}

$\mathbf{}$

\subsection*{Notations} We work over complex numbers $\mathbb{C}$. A $curve$ means a smooth irreducible projective curve.

\subsection*{Acknowledgments} The author would like to thank his advisor Herb Clemens for continuous support and helpful conversations. 

\section{Necessary conditions for curves to be isolated in smooth quintic threefolds}

\begin{lemma}
Let $C\subset \mathbb{P}^4$ be a curve and $Y \subset \mathbb{P}^4$ be a smooth quintic 3-fold. If $C\subset Y$ and $C$ is isolated in $Y$, then $h^i(\mathcal{N}_{C/\mathbb{P}^4})=h^i(\mathcal{O}_C(5)), \; i=0,1$.
\end{lemma}

\begin{proof}
Suppose $C$ is isolated in $Y$, then $h^0(\mathcal{N}_{C/Y})=h^1(\mathcal{N}_{C/Y})=0$. Considering  $$0\rightarrow \mathcal{N}_{C/Y}\rightarrow \mathcal{N}_{C/\mathbb{P}^4}\rightarrow \mathcal{O}_C(5)\rightarrow 0$$ Taking cohomology groups, it's easy to see $H^i(\mathcal{N}_{C/\mathbb{P}^4})\cong H^i(\mathcal{O}_C(5)),\; i=0,1$
\end{proof}
 
\begin{remark}
Lemma 2.1 gives us a useful necessary condition for curves to be isolated in smooth quintic threefolds. Actually, as we will see later, it turns out to be a rather strong condition.
\end{remark}

\begin{lemma}
$C\subset \mathbb{P}^n$ a curve. Suppose $C$ is degenerate, i.e., $C$ is contained in a hyperplane. Then $5h^1(\mathcal{O}_C(1))\geq h^1(\mathcal{N}_{C/{\mathbb{P}^n}})\geq h^1(\mathcal{O}_C(1)).$ In particular, $h^1(\mathcal{N}_{C/{\mathbb{P}^n}})=0$ if and only if $h^1(\mathcal{O}_C(1))=0$.
\end{lemma}

\begin{proof}
Notice that we have the following two exact sequences: $$0\rightarrow \mathcal{O}_C\rightarrow \mathcal{O}_C(1)^{(n+1)}\rightarrow \mathcal{T}_{\mathbb{P}^n}|_C\rightarrow 0, $$  $$0\rightarrow \mathcal{T}_C\rightarrow \mathcal{T}_{\mathbb{P}^n}|_C\rightarrow \mathcal{N}_{C/\mathbb{P}^n}\rightarrow 0$$
Then clearly, $5h^1(\mathcal{O}_C(1))\geq h^1(\mathcal{N}_{C/\mathbb{P}^n})$.

On the other hand, we have the following exact sequence: $$0\rightarrow \mathcal{N}_{C/\mathbb{P}^{n-1}}\rightarrow \mathcal{N}_{C/\mathbb{P}^n}\rightarrow \mathcal{O}_C(1)\rightarrow 0$$ 

Obviously, $h^1(\mathcal{N}_{C/\mathbb{P}^n})\geq h^1(\mathcal{O}_C(1))$.

\end{proof}

For later purpose, we need to generalize Lemma 2.3 to get a lower bound for $h^1(\mathcal{N}_{C/{\mathbb{P}^n}})$. To this end , we need to show the following lemma:

\begin{lemma}
$X\subset \mathbb{P}^n$  reduced and irreducible variety. Let $d$ be the smallest integer such that $h^0(\mathcal{I}_X(d))\neq 0$, where $\mathcal{I}_X$ is the ideal sheaf of $X$. Then $\forall F\in H^0(\mathbb{P}^n, \mathcal{I}_X(d))$, $F$ is an irreducible homogeneous polynomial and the singular locus of $V(F)$ doesn't contain $X$, where $V(F)$ is the variety defined by $F$. 
\end{lemma}

\begin{proof}
If $F$ is not irreducible, then $X$ must be contained in a hypersurface of degree less than $d$, but that is impossible by the definition of $d$. Similarly, the singular locus of $V(F)$ is defined by polynomails of degree $d-1$(more explicitly, partial derivatives of $F$), so $X$ is not contained in the singular locus of $V(F)$.
 
\end{proof}

The following lemma is critical to the rest of this paper because it gives a nice lower bound for $h^1(\mathcal{N}_{C/{\mathbb{P}^n}})$.

\begin{lemma}
$C\subset \mathbb{P}^n$ a curve. Let $m$ be the smallest integer such that $h^0(\mathcal{I}_C(m))\neq 0$. Then $h^1(\mathcal{N}_{C/{\mathbb{P}^n}})\geq h^1(\mathcal{O}_C(m))$.
\end{lemma}

\begin{proof}
Let $F\in H^0(\mathbb{P}^n, \mathcal{I}_C(m))$, and $Y:=V(F)$. Considering the following exact sequence of ideal sheaves : $$0\longrightarrow \mathcal{I}_{Y/{\mathbb{P}^n}}\longrightarrow \mathcal{I}_{C/{\mathbb{P}^n}}\longrightarrow \mathcal{I}_{C/{Y}}\longrightarrow 0$$

Restricting the above exact sequence to $C$ (i.e. tensoring $\mathcal{I}_{C/{\mathbb{P}^n}}$) $$0\longrightarrow \frac{\mathcal{I}_{Y/{\mathbb{P}^n}}}{\mathcal{I}_{Y/{\mathbb{P}^n}}\mathcal{I}_{C/{\mathbb{P}^n}}}\stackrel{\phi}{\longrightarrow} \frac{\mathcal{I}_{C/{\mathbb{P}^n}}}{\mathcal{I}_{C/{\mathbb{P}^n}}^2}\longrightarrow \frac{\mathcal{I}_{C/{Y}}}{\mathcal{I}_{C/{Y}}^2}\longrightarrow 0$$

Notice that $\phi$ is injective because of Lemma 2.4. Actually, $\phi$ is obviously injective at the points where $Y$ is smooth, so $\phi$ is injective generically by Lemma 2.4. Then $\phi$ is injective everywhere because $\frac{\mathcal{I}_{Y/{\mathbb{P}^n}}}{\mathcal{I}_{Y/{\mathbb{P}^n}}\mathcal{I}_{C/{\mathbb{P}^n}}}$ is locally free. 

Apply $\mathcal{H}om_{\mathcal{O}_C}(\_\_, \mathcal{O}_C)$ to the above exact sequence: $$0\rightarrow \mathcal{N}_{C/Y}\rightarrow \mathcal{N}_{C/\mathbb{P}^n}\rightarrow \mathcal{N}_{Y/\mathbb{P}^n}|_C\rightarrow \mathcal{E}xt_{\mathcal{O}_C}^1(\frac{\mathcal{I}_{C/{Y}}}{\mathcal{I}_{C/{Y}}^2}, \mathcal{O}_C)\rightarrow 0$$

$ \mathcal{E}xt_{\mathcal{O}_C}^1(\frac{\mathcal{I}_{C/{Y}}}{\mathcal{I}_{C/{Y}}^2}, \mathcal{O}_C)$ is a torsion sheaf and hence $H^1(C, \mathcal{E}xt_{\mathcal{O}_C}^1(\frac{\mathcal{I}_{C/{Y}}}{\mathcal{I}_{C/{Y}}^2}, \mathcal{O}_C))=0$. Then it is easy to see $h^1(\mathcal{N}_{C/{\mathbb{P}^n}})\geq h^1(\mathcal{N}_{Y/{\mathbb{P}^n}}|_C)=h^1(\mathcal{O}_C(m))$.

\end{proof}

\begin{corollary}
$C\subset \mathbb{P}^n$ a curve. Suppose $C$ is contained in a hypersurface of degree $d$, then $h^1(\mathcal{N}_{C/{\mathbb{P}^n}})\geq h^1(\mathcal{O}_C(d)).$
\end{corollary}

The following theorem explains why $h^1(\mathcal{N}_{C/\mathbb{P}^4})=h^1(\mathcal{O}_C(5))$ is a rather strong condition for a curve $C\subset \mathbb{P}^4$ and essentially, it is one of the main ingredients of the proof of the non-existence results, namely, Theorem 3.5.

\begin{theorem}
$C\subset \mathbb{P}^4$ a curve. Suppose $C$ is contained in a hypersuface of degree $d\leq 4$. Then $h^1(\mathcal{N}_{C/\mathbb{P}^4})=h^1(\mathcal{O}_C(5))$ if and only if $h^1(\mathcal{N}_{C/\mathbb{P}^4})=h^1(\mathcal{O}_C(5))=h^1(\mathcal{O}_C(d))=0$.
\end{theorem}

\begin{proof}
The ``if'' part is trivial, so we just need to show the `` only if '' part. Suppose we have $h^1(\mathcal{N}_{C/\mathbb{P}^4})=h^1(\mathcal{O}_C(5))$.  By Corollary 2.6, $h^1(\mathcal{N}_{C/\mathbb{P}^4})\geq h^1(\mathcal{O}_C(d))$ and hence $h^1(\mathcal{O}_C(5))\geq h^1(\mathcal{O}_C(d)).$ 

But $h^1(\mathcal{O}_C(d))=h^0(\mathcal{K}_C(-d))=h^0(\mathcal{K}_C(-5)\otimes\mathcal{O}_C(5-d))$, where $\mathcal{K}_C$ is the canonical bundle of $C$. If $h^0(\mathcal{K}_C(-5))=0$, we are done because $0=h^0(\mathcal{K}_C(-5))=h^1(\mathcal{O}_C(5))$. If $h^0(\mathcal{K}_C(-5))\neq 0$, by [2, Ch IV, Lemma 5.5] $h^0(\mathcal{K}_C(-5))+h^0(\mathcal{O}_C(5-d))\leq h^0(\mathcal{K}_C(-d))+1$. But $5-d\geq 1$, so $h^0(\mathcal{O}_C(5-d))\geq 2$, and hence $h^0(\mathcal{K}_C(-5))+1\leq h^0(\mathcal{K}_C(-d))$. Thus, by Serre duality we have $h^1(\mathcal{O}_C(5))+1\leq h^1(\mathcal{O}_C(d))$, contradicting $h^1(\mathcal{O}_C(5))\geq h^1(\mathcal{O}_C(d))$.
\end{proof}

\begin{remark}
Theorem 2.7 tells us that if a curve $C\subset \mathbb{P}^4$ is isolated in a smooth quintic threefold and $C$ is contained in some hypersurface of degree $\leq 4$, then $C$ is even unobstructed as a curve in $\mathbb{P}^4$ (more precisely, $h^1(\mathcal{N}_{C/{\mathbb{P}^4}})=0$) and hence $[C]\in Hilb(\mathbb{P}^4)$ is a smooth point (Cf.[4, Ch.I, $\S$ 1.2] ).
\end{remark}

Let $C\subset \mathbb{P}^n$ be a curve of degree $d$ and genus $g$. Let $5>k> 0$ be an integer. By Riemann-Roch, $h^1(\mathcal{O}_C(k))=h^0(\mathcal{O}_C(k))-kd-1+g$, this means, roughly speaking, if $g$ is ``very big'' with respect to $d$ (for example, $g>kd+1$), then $h^1(\mathcal{O}_C(k))$ will be positive. Furthermore, if we hope $C$ to satisfy $h^1(\mathcal{N}_{C/\mathbb{P}^4})=h^1(\mathcal{O}_C(5))$, then by Theorem 2.7 $C$ can not be contained in a hypersurface of degree $\leq k$. More precisely, we have the following:

\begin{theorem}
$C\subset\mathbb{P}^4$ a curve. $C$ is not contained in any plane (i.e. two dimensional linear subspace of $\mathbb{P}^4$) and has degree $d$ and genus $g$. Suppose $h^1(\mathcal{N}_{C/{\mathbb{P}^4}})=h^1(\mathcal{O}_C(5))$. Then: \\ (i) If $g>d-3$ and  $d\geq 3$, then $C$ is non-degenerate, i.e. $H^0(\mathbb{P}^4, \mathcal{I}_C(1))=0$. \\ (ii) If $g>2d-11$ and $d\geq 8$, then $C$ is not contained in any quadric hypersurfaces;\\ (iii) If $g>3d-18$ and $d\geq 8$, then $C$ is not contained in any cubic hypersurfaces.
\end{theorem}

\begin{proof}
(i): Assume $g>d-3$ and $d\geq 3$. Suppose $C$ is degenerate, then $h^0(\mathcal{I}_C(1))=1$ because $C$ is not in any plane. By R-R,  $h^1(\mathcal{O}_C(1))=h^0(\mathcal{O}_C(1))-d-1+g\geq 4-d-1+g=g-d+3>0$. On the other hand, by Theorem 2.7 $h^1(\mathcal{O}_C(1))=0$, contradiction. Therefore, $C$ is non-degenerate.

(ii):Assume $g>2d-11$ and $d\geq 8$. Suppose $C$ is contained in a quadric hypersurface. First of all, when $d\geq 8$, $2d-11\geq d-3$, so by (i) $C$ is non-degenerate. Then by [7, Corollary 1.5], $h^0(\mathcal{I}_C(2))\leq 15-2(4+1)-2$, so $h^0(\mathcal{O}_C(2))\geq 12$. By R-R again, $h^1(\mathcal{O}_C(2))=h^0(\mathcal{O}_C(2))-2d-1+g\geq 12-2d-1+g=g-2d+11>0$, contradiction by Theorem 2.7.

(iii): Assume $g>3d-18$ and $d\geq 8$. Suppose $C$ is contained in a cubic hypersurface. By (ii) $C$ can not be in a quadric hypersurface, so $h^0(\mathcal{I}_C(1))=h^0(\mathcal{I}_C(2))=0$. Therefore $h^0(\mathcal{O}_C(1))\geq 5$ and $h^0(\mathcal{I}_C(2))\geq 15$. Then by [2, Ch. IV, Lemma 5.5] $h^0(\mathcal{O}_(3))\geq 19$. So $h^1(\mathcal{O}_C(3))=h^0(\mathcal{O}_C(3))-3d-1+g\geq 19-3d-1+g=g-3d+18>0$, again contradiction by Theorem 2.7.
\end{proof}

\begin{remark}
It is possible that there are better ways to estimate either $h^0(\mathcal{I}_C(k))$ or $h^1(\mathcal{I}_C(k))$. If that is the case, the results in Theorem 2.9 may be improved.
\end{remark}

\section{Castelnuovo theory and non-existence of smooth isolated curves in quintic threefolds}

Let $C\subset \mathbb{P}^n$ be a curve. $C$ has degree $d$ and genus $g$. Roughly speaking, Castelnuovo theory tells us that if the $g$ is `` large'' with respect to $d$, $C$ has to be contained in surfaces/hypersurfaces of ``small'' degree. More precisely, in the case of $n=4$, we have the following:

\begin{theorem}
([1, Theorem 3.7, Theorem 3.15 and Theorem 3.22]) Let $C\subset \mathbb{P}^4$ be a curve of degree $d$ and genus $g$. Then:

(i) If $g>\frac{d^2-5d+6}{6}$ and $d\geq 3$, then $C$ is degenerate.

(ii) If $C$ is non-degenerate, $g>\frac{d^2-4d+8}{8}$ and $d\geq 9$, then $C$ is contained in a non-degenerate irreducible surface of degree 3.

(iii) If $C$ is non-degenerate, $g>\frac{d^2-3d+10}{10}$, and $d\geq 144$, then $C$ is contained in a non-degenerate irreducible surface of degree 4 or less. 
\end{theorem}

If we want to use Theorem 2.9 to get some non-existence results, roughly speaking, we need to show that if the genus $g$ is ``large'' with respect to degree $d$ then the curve $C\subset \mathbb{P}^4$ has to be contained in a ``low''  degree hypersurface. But Theorem 3.1(ii) and (iii) only tell us that curves with ``large'' genera are contained in ``low'' degree surfaces. Therefore, we need to show that ``low'' degree surfaces has to be contained in ``low'' degree hypersurfaces. Fortunately, we have the following:

\begin{lemma}
([9, Lemma 3]) Let $W\subset \mathbb{P}^n$ be an irreducible non-degenerate variety of dimension $m$ and degree $d$. Let $A\in W$ be a closed point; and if $W$ is a cone suppose that $A$ is not a vertex of $W$. Let $W_1$ be the cone obtained by joining $A$ to every point of $W$. Then $W_1$ does not lie in any hyperplane of $\mathbb{P}^n$, and it has dimension exactly $m+1$ and degree at most $d-1$; moreover, if $A$ is a singular point of $W$ then $W_1$ has degree at most $d-2$. 
\end{lemma}

Now the following is just an easy consequence of Lemma 3.2.

\begin{proposition}
Let $X\subset \mathbb{P}^4$ be a non-degenerate irreducible surface of degree $d$. Then $X$ is contained in a hypersurface of degree $d-1$; moreover, if $X$ has a singular point which is not a vertex of $X$, then $X$ is contained in a hypersurface of degree $d-2$.
\end{proposition}

\begin{proof}
Let $A\in X$ be a closed point, and if $X$ is a cone suppose $A$ is not a vertex of $X$. Let $X_1$ be the cone obtained by joining $A$ to every point of $X$. By Lemma 3.2, $X_1$ is an irreducible non-degenerate threefold of degree at most $d-1$( $d-2$ if $A$ is a singular point of $X$). Notice that projection of $X$ from $A$ to a hyperplane $H$ not containing $A$ is exactly equal to the intersection $X_1\cap H$ which is a surface in $H$ of degree at most $d-1$( $d-2$ if $A$ is a singular point of $X$). Obviously, any surface of degree at most $d-1$ ( $d-2$ ) in $H\cong \mathbb{P}^3$ is the zero locus of a polynomial of degree at most $d-1$ ( $d-2$ ), so we are done.
\end{proof}

\begin{remark}
Notice if that the surface $X$ in Proposition 3.3 is smooth, $X$ is even $d-1$-regular and hence the homogeneous ideal of $X$ is even generated by polynomials of degree $d-1$ or less ( Cf. [5] ).
\end{remark}

Finally, we are ready to prove the following non-existence results:

\begin{theorem}
Let $d\geq 3$ and $g\geq 0$ be integers. Let $C\subset \mathbb{P}^4$ be a curve of degree $d$ and genus $g$. Then $C$ can not be isolated in any smooth quintic threefolds if the pair $(d,\; g)$ is in the following list: 

(i) $g>d-3$, $(d,g)\neq (3,1)$ and $3\leq d\leq 8$;

(ii) $g>2d-11$ and $9\leq d\leq 12$;

(iii) $g>\frac{d^2-4d+8}{8}$ and $12<d<144$;

(iv) $g>\frac{d^2-3d+10}{10}$ and $d\geq 144$.

\end{theorem}

\begin{proof}
(i) Assume $g>d-3$, $(d,g)\neq (3,1)$ and $3\leq d\leq 8$. Notice that when $3\leq d\leq 8$, $d-3\geq \frac{d^2-5d+6}{6}$, so  $g>\frac{d^2-5d+6}{6}$, then by Theorem 3.1(i) $C$ is contained in a hyperplane. Therefore, by Theorem 2.9(i) $C$ has to be contained in a plane. But it is easy to check that if $C$ is contained in a plane, $h^1(\mathcal{N}_{C/\mathbb{P}^4})=h^1(\mathcal{O}_C(5))$ only if $(d,g)=(3,1)$. But by assumption $(d,g)\neq (3,1)$, so $C$ can not be isolated in any smooth quintic threefolds by Lemma 2.1.

(ii) Assume  $g>2d-11$ and $9\leq d\leq 12$. Notice that in this case, $g>\frac{d^2-4d+8}{8}$, then by Theorem 3.1(ii) and Proposition 3.3, $h^0(\mathcal{I}_C(2))\neq 0$. Thus, by Theorem 2.9(ii) $h^1(\mathcal{N}_{C/\mathbb{P}^4})\neq h^1(\mathcal{O}_C(5))$, so $C$ can not be isolated in any smooth quintic threefolds by Lemma 2.1.

(iii) Assume $g>\frac{d^2-4d+8}{8}$ and $12<d<144$. Notice that in this case $\frac{d^2-4d+8}{8}\geq 2d-11$, then the rest of the argument is similar to case (ii).

(iv) Similarly as in cases (ii) and (iii).

\end{proof}

More intuitively, we can see non-existence/existence of smooth isolated curves in general quintic threefolds from the following figure.

\center{Figure 1}

\includegraphics{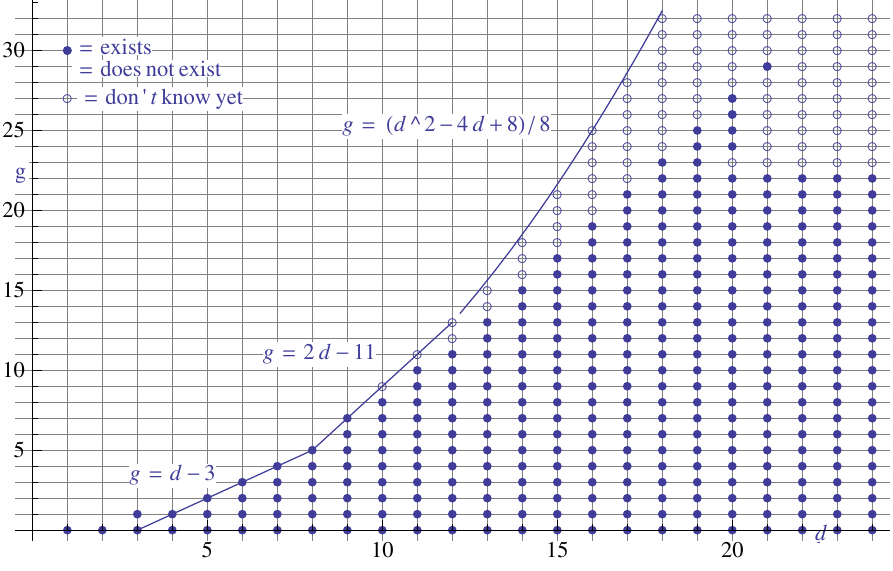}

\vspace{2mm}

As an application of Theorem 3.5, we get the following:

\begin{corollary}
If there exists a smooth isolated curve of degree $d\leq 9$ and genus $g$ in a general quintic threefold, then the pair of integers $(d, g)$ must be in Knutsen's list in [3, Theorem 1.2]. In other words, Knutsen's list [3, Theorem 1.2] is complete for $Y=(5)\subset \mathbb{P}^4$ and $d\leq 9$.
\end{corollary}

\end{document}